\documentclass{amsart}
\usepackage[paper,pkg/amsaddr=true,macros/file=false]{zbs}
\title{Schur complements for tensors and multilinear commutative rank}
\author[Guy Moshkovitz]{Guy Moshkovitz\nfts{1}}
\address{\nfts{1}Department of Mathematics, City University of New York (Baruch College \& Graduate Center), New York, NY 10010, USA}
\email{guymoshkov@gmail.com}
\author[Daniel G. Zhu]{Daniel G. Zhu\nfts{2}}
\address{\nfts{2}Department of Mathematics, Princeton University, Princeton, NJ 08544, USA}
\email{zhd@princeton.edu}

\thanks{The first author is supported by NSF Award DMS-2302988. The second author is supported by the NSF Graduate Research Fellowships Program (NSF grant DGE-2039656).}

\newcommand{\mat}[1]{{#1}_\square}

\DeclareMathOperator{\bias}{bias}
\DeclareMathOperator{\mindeg}{mindeg}
\DeclareMathOperator{\mult}{mult}
\DeclareMathOperator{\rank}{rank}
\DeclareMathOperator{\PR}{PR}
\DeclareMathOperator{\CR}{CR}
\DeclareMathOperator{\LR}{LR}
\DeclareMathOperator{\AR}{AR}
\DeclareMathOperator{\SR}{SR}
\DeclareMathOperator{\MaxR}{MR}

\DeclareMathOperator{\codim}{codim}

\newcommand{\setk}{\mathbb{K}}

\newcommand{\calr}{\mathcal{R}}
\newcommand{\cals}{\mathcal{S}}
\newcommand{\calm}{\mathcal{M}}

\newcommand{\frakp}{\mathfrak{p}}
\newcommand{\frakq}{\mathfrak{q}}
\newcommand{\fraka}{\mathfrak{a}}

\newcommand{\vp}{{\bm{p}}}
\newcommand{\vx}{{\bm{x}}}
\newcommand{\vy}{{\bm{y}}}
\newcommand{\vz}{{\bm{z}}}

\numberwithin{equation}{section}

\begin{document}
\begin{abstract}
We show that three notions of rank for matrices of multilinear forms are equivalent. This result generalizes a classical result of Flanders, corrects a minor hole in work of Fortin and Reutenauer, answers a question of Lampert on the relation between the analytic and slice ranks of trilinear forms, and establishes a special case of the conjecture that the analytic and partition ranks of a tensor are equivalent.
\end{abstract}
\maketitle

\section{Introduction}
Fix any field\footnote{While some parts of this paper assume $\setf$ is finite, most results apply over all fields. This is in contrast to some other results (e.g.\ in~\cite{CM22,Der22,CM23,CY25}) which assume that $\setf$ is perfect.} $\setf$ 
and an integer $d \geq 0$. Let $x_1 = (x_{1,1},\ldots,x_{1,n})$ through $x_d = (x_{d,1},\ldots,x_{d,n})$ be $d$ collections of variables, where we let $n$ be an unbounded parameter. Furthermore, let $\calm_d$ denote the space of multilinear forms over $\setf$ that are individually linear in $x_1,\ldots,x_d$, i.e.\ polynomials of the form
\[\sum_{i_1,\ldots,i_d = 1}^n a_{i_1,\ldots,i_d} x_{1,i_1} \cdots x_{d, i_d},\]
where $a_{i_1,\ldots,i_d}$ are coefficients in $\setf$.

Suppose $M(x_1,\ldots,x_d)$ is an $\calm_d$-matrix, where for the rest of this paper we define an $L$-matrix to be a matrix with entries in $L$.
We consider three ways to define the rank of $M$:
\begin{itemize}
    \item the \vocab{max-rank} $\MaxR(M)$, defined to the maximum possible rank of the evaluation $M(x_1,\ldots,x_d)$ over all choices of $x_1,\ldots,x_d \in \setf^n$.
    \item the \vocab{commutative rank} $\CR(M)$, defined to be the rank of $M$ when viewed as a matrix over the field of rational functions in the variables $x_{i,j}$. Equivalently, this is the maximum $r$ for which an $r \times r$ minor of $M$ is a nonzero polynomial, or the maximum rank of the evaluation $M(x_1,\ldots,x_d)$ over all choices of $x_1,\ldots,x_d \in \overline{\setf}^n$. ($\overline{\setf}$ can be replaced with any infinite extension of $\setf$.)
    \item the \vocab{partition rank} $\PR(M)$, defined to be the minimum number of multilinear outer products needed to sum to $M$. Specifically, given a set $S \subseteq [d] \coloneq \set{1,2,\ldots,d}$ and multilinear vectors $u((x_i)_{i \in S})$ and $v((x_i)_{i \in [d] \setminus S})$, the outer product $u \otimes v$ is an $\calm_d$-matrix. The partition rank is the minimum $r$ such that $M$ is the sum of $r$ such matrices. (The choice of $S$ can vary between the matrices in the sum.)
\end{itemize}
The case $d = 0$ corresponds to the case of matrices of scalars, and it is an elementary fact of linear algebra that in this case these ranks are all equal to the standard notion of rank. For general $d$, it is easy to see that 
\[\MaxR(M) \leq \CR(M) \leq \PR(M),\]
but these notions are all distinct. For example, if $d = 1$, the matrix
\[\begin{pmatrix}
\alpha & 0 & 0 \\
0 & \beta & 0 \\
0 & 0 & \alpha+\beta
\end{pmatrix}\]
has max-rank $2$ but commutative rank $3$ over the field $\setf_2$, while the matrix
\[\begin{pmatrix}
0 & \alpha & \beta \\
-\alpha & 0 & \gamma \\
-\beta & -\gamma & 0
\end{pmatrix}\]
has commutative rank $2$ (as its determinant is zero)
but partition rank $3$. (Here $\alpha,\beta,\gamma$ are shorthands for $x_{1,1}, x_{1,2}, x_{1,3}$, respectively.)

In this paper, we show that these three notions of rank are equivalent up to constant factors:
\begin{thm} \label{thm:main}
We have
\[\PR(M) \leq (2^d + O_d(\abs{\setf}\inv)) \CR(M) \quad \text{and} \quad \CR(M) \leq \paren*{1 - \frac{1}{\abs{\setf}}}^{-d} \MaxR(M).\]
Here $1/\abs{\setf}$ should be interpreted as zero if $\setf$ is infinite.
\end{thm}
The constant $(1-1/\abs{\setf})^{-d}$ is tight (see \cref{prop:tight}). It is known that the constant $2^d$ is tight for $d \leq 1$ \cite{DM18}, but this is unknown for higher $d$.

\subsection{Linear matrices}
In the case $d = 1$, the ranks we consider reduce to a number of foundational notions. The problem of understanding linear matrices of low max-rank is equivalent to understanding vector spaces of low-rank matrices;
indeed, evaluating a linear matrix at all points gives a vector space of matrices, and vice versa. 
Linear matrices of bounded max-rank have been the subject of a long line of work (see e.g.~\cite{EH88} and references therein), going back to work of Dieudonn\'e \cite{Die49}. Although an exact classification of such spaces appears to be intractable \cite[Sec.~4]{EH88}, a basic result in the area is the following result of Flanders:
\begin{thm}[Flanders \cite{Fla62}] \label{thm:flanders}
Let $M$ be an $\calm_1$-matrix and suppose $\abs{\setf} > \MaxR(M) = r$. Then there exist invertible scalar matrices $P$ and $Q$ such that
\[PMQ = \begin{pmatrix}
M_{11} & M_{12} \\
M_{21} & 0
\end{pmatrix}\]
where $M_{11}$ is of size $r \times r$.
\end{thm}
Following Eisenbud and Harris \cite{EH88} (see also \cite{Atk81}), if we may find invertible $P$ and $Q$ such that
\[PMQ = \begin{pmatrix}
M_{11} & M_{12} \\
M_{21} & 0
\end{pmatrix}\]
with $M_{11}$ of size $r_1 \times r_2$, we call $M$ an \vocab{$(r_1+r_2)$-compression space}; observe that every $r$-compression space has commutative rank and hence max-rank at most $r$. It is straightforward to show that $\PR(M)$ is the smallest value of $r$ for which $M$ is an $r$-compression space, so \cref{thm:flanders} implies that over large fields, we have $\PR(M) \leq 2\MaxR(M)$, essentially the statement of \cref{thm:main} for $d = 1$. The constant $2$ has since been shown to be tight \cite{DM18}.

In part since most of the early work on spaces of low-rank matrices (with the notable exception of \cite{Mes85}) worked over fields large enough that the max-rank and commutative rank coincide, the concept of commutative rank as an inherently polynomial property developed largely independently in the context of algebraic complexity theory, where the problem of computing the commutative rank is a classical problem of Edmonds \cite{Edm67}. The term ``commutative rank'' itself is significantly newer, originating from work of Fortin and Reutenauer \cite{FR04} comparing it to the \vocab{noncommutative rank} (first defined in \cite{Ber67}, see also \cite{Coh71,Coh95}), where the $x_{1,i}$ are treated as free noncommuting variables. For matrices $M$ of linear forms, Fortin and Reutenauer showed that $M$ having noncommutative rank at most $r$ is equivalent to $M$ being an $r$-compression space, i.e.\ the noncommutative rank and the partition rank are the same. Applying \cref{thm:flanders}, Fortin and Reutenauer concluded that the noncommutative rank is at most twice the commutative rank, though they seemingly neglected to mention the field size restriction in \cref{thm:flanders}. Our techniques prove this unconditionally:
\begin{cor} \label{thm:d1bound}
For an $\calm_1$-matrix $M$, we have $\PR(M) \leq 2\CR(M)$.
\end{cor}

\subsection{Applications to tensors}

An interesting application of \cref{thm:d1bound} is to the ranks of $3$-tensors. Independent work of Adiprasito, Kazhdan, and Ziegler \cite{AKZ21} and Cohen and the first author \cite{CM22} established the following result:
\begin{thm}[Adiprasito-Kazhdan-Ziegler \cite{AKZ21}\footnote{In \cite{AKZ21}, it is claimed that an exact constant of $3$ suffices, with no dependence on $q$. However, we do not see how this is the case and believe it to be an error, especially since \cite{CM22} uses essentially the same proof and does not get this constant.}, Cohen-Moshkovitz \cite{CM22}] \label{thm:arsr}
If $\AR(T)$ and $\SR(T)$ denote the analytic and slice ranks of a $3$-tensor $T$ over a finite field $\setf_q$ (for definitions see \cref{sec:prelim}), we have
\[\SR(T) \leq \paren[\big]{3 + O(\log(q)\inv)} \AR(T).\]
\end{thm}
However, both proofs of \cref{thm:arsr} rely on work of Derksen \cite{Der22} in geometric invariant theory. Recently, Lampert \cite{Lam25} found an alternative elementary proof that $\SR(T) = O(\AR(T))$, showing that
\[\SR(T) \leq 5\AR(T) + O\paren[\big]{\log_{q}(\AR(T) + 1) + 1}.\]
(This bound can be boosted to $\SR(T) \leq 5\AR(T)$ by applying it to $T^{\oplus k}$ for large $k$ and using the fact that slice rank \cite{Gow21} and analytic rank add under direct sums; indeed, 
$\SR(T)/\AR(T) = \SR(T^{\oplus k})/\AR(T^{\oplus k}) \le 5 + o(1)$ as $k\to\infty$.) By using \cref{thm:d1bound}, we are able to recover the constant of $3$ using elementary methods, answering a question of Lampert \cite{Lam25} and improving the subleading term:
\begin{cor} \label{cor:arsr}
For any $3$-tensor $T$ over $\setf_q$,
\[\SR(T) \leq \paren*{3 + \frac{2}{q-1}} \AR(T).\]
\end{cor}
It is unclear if we should expect $3$ to be the optimal constant.

Beyond degree $1$, \cref{thm:main} can also be regarded as a special case of the conjecture \cite{Lov19,KZ21,AKZ21} that the analytic rank and the partition rank of tensors are equivalent. Indeed, an $\calm_d$-matrix $M$ naturally corresponds to a $(d+2)$-tensor $T$, and one can show that $\AR(T) \leq \MaxR(M)$ and $\PR(T) \leq \PR(M)$. In this language, our results are especially notable since the analytic rank has been recently proven to be equivalent to many other notions of tensor rank \cite{CY25,BL24}, while the partition rank has resisted such efforts. \cref{thm:main} appears to be the first result giving linear relations between the partition rank and other notions of rank over small fields and arbitrary degrees. For an extended discussion on the relation between \cref{thm:main} and ranks of tensors, we refer the reader to \cref{sec:arpr}.

\subsection{Proof overview}
The bulk of the proof of \cref{thm:main} lies in upper-bounding the partition rank of matrices of low commutative rank. To do so, we define a generalization of the Schur complement from linear algebra to tensors. As a refresher, suppose we have a matrix (of scalars) $M$ with a block decomposition
\[M = \begin{pmatrix}
A & B \\
C & D
\end{pmatrix}\]
where $A$ is invertible and of size $r \times r$. Then, the Schur complement $M/A \coloneq D - CA\inv B$ has the property that $\rank (M/A) = \rank M - r$. Moreover,
\[M - \begin{pmatrix}
0 & 0 \\
0 & M/A
\end{pmatrix} = \begin{pmatrix} A & B \\ C & CA\inv B \end{pmatrix} = \begin{pmatrix} A \\ C \end{pmatrix} A\inv \begin{pmatrix} A & B \end{pmatrix},\]
which is manifestly the sum of $r$ outer products. Therefore, this process allows us to write $M$ as the sum of outer products and a remainder term, given by the Schur complement, of lower rank.

Now suppose $M$ has entries in $\calm_d$. We may consider taking the Schur complement over the field of rational functions, yielding some matrix $M'$ such that $M - M'$ is the sum of a bounded number of outer products and $M'$ has lower (commutative) rank than $M$. However, in doing so a serious issue arises, which is that by inverting a submatrix we lose multilinearity. In particular, the factors in the outer product decomposition of $M - M'$, as well as the entries of $M'$ itself, become arbitrary rational functions. Fortunately, under certain circumstances, we are able to repair this issue, by approximating these rational functions with multilinear polynomials. By iterating this procedure, we are able to reduce the remainder term to zero, producing a full decomposition of $M$. These operations are similar to those that have appeared in previous work of the authors \cite{MZ22}, see \cref{sec:arpr} for more details.

A secondary key idea in the proof is the use of the so-called ``method of multiplicities'', where we deduce results from the fact that a polynomial of bounded degree can not vanish to high order at every point. This argument is used in both statements of \cref{thm:main}.

\subsection{Outline}
After some preliminary definitions in \cref{sec:prelim}, we define an approximation procedure in \cref{sec:approx} and use it to define a differential Schur complement in \cref{sec:sc}. This is used to prove \cref{thm:main} in \cref{sec:mainproof}. We prove \cref{thm:d1bound,cor:arsr} in \cref{sec:d1}; since \cref{sec:approx} of the paper is quite notationally heavy, we reprove everything in the $d=1$ case, where the arguments are significantly simpler. In particular, \cref{sec:d1} is mostly self-contained. We discuss the relation to tensor ranks of arbitrary degree in \cref{sec:arpr}.

\section{Preliminaries} \label{sec:prelim}
\subsection{Multi-homogeneous polynomials} 
For an unbounded parameter $n$, let $x_1 = (x_{1,1},\ldots,x_{1,n}), x_2 = (x_{2,1},\ldots,x_{2,n}),\ldots$ be a sequence of tuples of variables, and for a finite subset $S$ of the positive integers $\setn$, let $\calm_S$ be the set of multilinear forms (or multilinear homogeneous polynomials) over $\setf$ that are individually linear in $(x_i)_{i\in S}$. This is the degree-$(1,1,\ldots,1)$ component of the $\setz_{\geq 0}^S$-graded ring $\calr_S \coloneq \setf[x_{i,j} : i\in S, j\in [n]]$. For consistency with the introduction, we let $\calr_d$ and $\calm_d$ be abbreviations for $\calr_{[d]}$ and $\calm_{[d]}$, respectively.

Given $\vp_S = (p_i)_{i \in S} \in (\setf^n)^S$ and $f \in \calr_S$, we may evaluate the polynomial $f$ at $\vp_S$, yielding $f(\vp_S) \in \setf$. We will also require the concept of partial evaluations, where given $f \in \calr_S$ and $\vp_{S\setminus S'}=(p_i)_{i \in S\setminus S'} \in (\setf^n)^{S \setminus S'}$,
we obtain some polynomial $f[\vp_{S\setminus S'}] \in \calr_{S'}$ by setting $x_i = p_i$. The choice of $S'$ should always be clear from context.

Elements of $\setf^n$ are denoted with unbolded letters, while bolded letters denote collections of such elements. Given such a collection $\vx$, we always let $x_i$ be the $i$th element of $\vx$, and let $\vx_S$ be shorthand for $(x_i)_{i \in S}$.
As usual, $e_i=(0,\ldots,0,1,0,\ldots,0)$ denotes the $i$th standard basis vector.

\subsection{Matrices}
The main objects we study in this paper are $L$-matrices, i.e.\ matrices with entries in $L$, where $L$ is some $\setf$-vector space (like $\calm_d$) embedded in an $\setf$-algebra.\footnote{Every notion that we will define will be independent of the choice of $\setf$-basis, in the sense that it will be invariant/equivariant under changes of bases over $\setf$. In other words, an $a \times b$ $L$-matrix can be naturally viewed as an $\setf$-bilinear map $B \colon \setf^a \times \setf^b \to L$, and we will not use any structure of $\setf^a$ and $\setf^b$ beyond them being $a$- and $b$-dimensional vector spaces.}

Analogously, we define an $L$-vector to be a vector with entries in $L$, which can equivalently be viewed as an $\setf$-linear map $\setf^a \to L$. If $R$ is a ring, given two $R$-vectors $u$ and $v$, we may form their outer product $u \otimes v$, which is an $R$-matrix. In terms of linear maps, if $u$ and $v$ correspond to $\ell_1 \colon \setf^a \to R$ and $\ell_2 \colon \setf^b \to R$, then $u \otimes v$ corresponds to the bilinear map $(x,y) \mapsto \ell_1(x)\ell_2(y)$.

If $R$ is an integral domain and $M$ is an $R$-matrix, we let $\rank(M)$ denote the rank of $M$, when viewed as a matrix over the field of fractions of $R$. Note that for $M$ an $\calm_d$-matrix, $\CR(M) = \rank(M)$; we will use both notations depending on context.

Given a matrix $M$ and a function $\varphi$ on its entries, we let $\varphi(M)$ denote the matrix obtained by applying $\varphi$ entrywise. Note that in order for this to be basis-independent, it is both necessary and sufficient that $\varphi$ be $\setf$-linear.

\subsection{Tensors}
For the purposes of this paper, a $d$-tensor is an element of $\calm_d$.\footnote{$d$-tensors are naturally in bijection with $d$-dimensional arrays of scalars:
$\sum_{i_1,\ldots,i_d \in [n]} c_{i_1,\ldots,i_d} x_{1,i_1}\cdots x_{d,i_d}$ corresponds to $(c_{i_1,\ldots,i_d})_{i_1,\ldots,i_d \in [n]}$.}
For a $d$-tensor $T \in \calm_d$, we define the following notions of rank:
\begin{itemize}
\item If $\setf = \setf_q$, define the \vocab{analytic rank} $\AR(T) = -\log_{q}\bias(T)$, 
where $\bias(T) = \Pr_{\vp \in (\setf^n)^{d-1}}[T[\vp]=0]$. Here, and for the rest of the paper, $\Pr_a$ and $\sete_a$ denote the probability and expectation, respectively, taken over a uniformly random choice of $a$.
\item The \vocab{partition rank} $\PR(T)$ is the minimum $r$ such that $T$ can be written as a sum of $r$ reducible tensors, i.e.\ tensors of the form $T_1 T_2$, where $T_1 \in \calm_S$ and $T_2 \in \calm_{S'}$ for some nontrivial partition $S \sqcup S' = [d]$. (The choice of partition may vary.)
\item The \vocab{slice rank} $\SR(T)$ is defined similarly to partition rank, except it requires $\deg(T_1)=1$ (so $\abs{S} = 1$).
\end{itemize}
For $d=3$, partition rank and slice rank are equal since every nontrivial bipartition of $\set{1,2,3}$ has a singleton.

If $d \geq 2$, any tensor $T \in \calm_d$ can be associated with an $\calm_{d-2}$-matrix, which we denote $\mat{T}$. Specifically, $\mat{T} = (m_{ij})_{i,j=1}^n$, where $m_{ij} \in \calm_{d-2}$ are the unique choices such that
\[T = \sum_{i,j=1}^n m_{ij} x_{d-1,i} x_{d,j}.\]
Put differently, $T(x_1,\ldots,x_d) = x_{d-1}^t \mat{T}(x_1,\ldots,x_{d-2}) x_{d}$.

We caution the reader that the term ``partition rank'' is used for two slightly different notions of rank. Converted back into the notion of tensors, $\PR(\mat{T})$ is defined similarly to the partition rank of $T$, except that in each product, $|S \cap \{d-1,d\}| = 1$.
So $\PR(\mat{T}) \geq \PR(T)$, but equality may not hold.

\section{Approximation} \label{sec:approx}
In this section we will describe a procedure to approximating rational functions by multilinear forms, which will be essential in defining the differential Schur complement. This section is the most technical in the paper.

\subsection{Approximating rational functions}
For $\vp \in (\setf^n)^d$, let $\calr_{d,\vp}$ be the localization of $\calr_d$ obtained by inverting all homogeneous polynomials that do not vanish at $\vp$, i.e.\ the ring of all homogeneous rational functions in $\setf(x_1,\ldots,x_d)$ that are defined at $\vp$.
Note that this inherits a $\setz^d$-grading from $\calr_d$, by letting $\deg(A/B) = \deg A - \deg B \in \setz^d$.
Let $\calm_{d,\vp}$ denote the degree-$(1,1,\ldots,1)$ component of $\calr_{d,\vp}$. We now define an approximation map based on taking $d$ directional derivatives at $\vp$.
\begin{defn} \label{def:approx}
The map $[-]_\vp \colon \calm_{d,\vp} \to \calm_d$ is defined so that
\[[f]_\vp (\vy) = \left.\frac{\partial}{\partial t_1} \cdots \frac{\partial}{\partial t_d} f(p_1 + t_1y_1, \ldots, p_d + t_dy_d)\right|_{t_1 = \cdots = t_d = 0}.\]
Equivalently,
\[[f]_\vp = \sum_{i_1,\ldots,i_d \in [n]} c_{i_1,\ldots,i_d} x_{1,i_1}\cdots x_{d,i_d} \in \calm_d\]
where
\[c_{i_1,\ldots,i_d}=\frac{\partial^d f}{\partial x_{1,i_1}\cdots \partial x_{d,i_d}}(\vp).\]
\end{defn}

Note that if $f \in \calm_d$, $[f]_\vp(\vy)$ is the coefficient of $t_1\cdots t_d$ in $f(p_1 + t_1y_1,\ldots,p_d + t_dy_d)$, which is just $f(\vy)$. In particular, $[[f]_\vp]_\vp = [f]_\vp$.

\begin{examp} \label{ex:approx}
Suppose $d = 1$ and
\[f = \frac{\alpha_1^2 + \alpha_2^2}{\alpha_1 + \alpha_2} \in \calm_{1,e_1},\]
where $\alpha_i$ is short for $x_{1,i}$. Then
\[\frac{\partial f}{\partial \alpha_1} = \frac{\alpha_1^2 + 2\alpha_1\alpha_2 - \alpha_2^2}{(\alpha_1+\alpha_2)^2} \quad\text{and}\quad \frac{\partial f}{\partial \alpha_2} = \frac{-\alpha_1^2 + 2\alpha_1\alpha_2 + \alpha_2^2}{(\alpha_1+\alpha_2)^2},\]
which evaluate to $1$ and $-1$ at $e_1$. Therefore $[f]_{e_1} = \alpha_1 - \alpha_2$.
\end{examp}

\cref{def:approx} is suprisingly difficult to work with as it does not treat $\calm_d$ as a subset of $\calm_{d,\vp}$, so will henceforth use an alternate definition of the approximation map. To formulate it, for $i \in [d]$, let $\frakp_i \subseteq \calr_{d, \vp}$ denote the ideal generated by homogeneous rational elements vanishing at $x_i = p_i$. 
\begin{prop} \label{prop:eltapprox}
For $f \in \calm_{d,\vp}$, the approximation $[f]_\vp$ is the unique $\bar f \in \calm_d$ such that $f - \bar f \in \sum_{i\in [d]} \frakp_i^2$.
\end{prop}
\begin{examp}
If $f$ is as in \cref{ex:approx}, then
\[f - [f]_{e_1} = \frac{\alpha_1^2 + \alpha_2^2}{\alpha_1 + \alpha_2} - \frac{\alpha_1^2 - \alpha_2^2}{\alpha_1 + \alpha_2} = \frac{2\alpha_2^2}{\alpha_1+\alpha_2},\]
which is in $\frakp_1^2$ as $\alpha_2 \in \frakp_1$.
\end{examp}

Before we prove \cref{prop:eltapprox}, we first work out a concrete description of $\calm_{d,\vp}$ which will be useful for future computations. For this description, assume that $p_i = e_1$ for all $i$. If we let 
\[z_{i} = (x_{i,2}/x_{i,1}, \ldots, x_{i,n}/x_{i,1})\]
for $i \in [d]$, every homogeneous polynomial in $\calr_d$ of degree $(k_1,k_2,\ldots,k_d)$ can be written as
\[x_{1,1}^{k_1} x_{2,1}^{k_2} \cdots x_{d,1}^{k_d} f(z_1,\ldots,z_d)\]
where $f$ is a (not necessary homogeneous) \emph{polynomial} of (total) degree at most $k_i$ in $z_i$. This polynomial is nonvanishing at $\vp$ if $f$ has a nonzero constant term. As a result, the degree-$(k_1,\ldots,k_d)$ elements of $\calr_{d,\vp}$ are of the form
\[x_{1,1}^{k_1} x_{2,1}^{k_2} \cdots x_{d,1}^{k_d}g(z_1,\ldots,z_d)\]
where $g$ is any rational function defined at $0$. From this characterization we may deduce $\frakp_i = \gen{z_{i,2},\ldots,z_{i,d}}$.

\begin{proof}[Proof of \cref{prop:eltapprox}]
Without loss of generality suppose $p_i = e_1$ for all $i$. We first show that an $\bar f$ exists with $f - \bar f \in \sum_{i\in [d]} \frakp_i^2$, and then show that any such $\bar f$ must be equal to $[f]_\vp$.

Let $g(z_1,\ldots,z_d) = f/(x_{1,1} \cdots x_{d,1})$ be a rational function defined at $z_i = 0$. We wish to find a polynomial $\bar g(z_1,\ldots,z_d)$ that is multi-affine (i.e.\ affine in each individual $z_i$) such that $g - \bar g \in \sum_{i \in [d]} \frakp^2_i$. If we find such a $\bar g$, we may let $\bar f = x_{1,1}\cdots x_{d,1} \bar g$.

Write $g = A/B$ where $A,B$ are polynomials in the ring $\calr' = \setf[z_{i,j}]$. Let $\fraka \subseteq \calr'$ be the ideal given by $\sum_i \gen{z_{i,2},\ldots,z_{i,d}}^2$. Since $\calr'/\fraka$ is a local ring with maximal ideal $\gen{z_{i,j}}$ and $B$ is not in this maximal ideal, there exists some $C \in \calr'$ such that $BC - 1 \in \fraka$. As a result,
\[g - AC = \frac{A - ABC}{B} \in \sum_{i\in [d]}\frakp_i^2\]
so $g$ is equivalent modulo $\sum_{i\in [d]}\frakp_i^2$ to some polynomial. All monomials of degree at least $2$ in some $z_i$ are in $\frakp_i^2$, so we may delete all such monomials to get a multi-affine $\bar g$, as desired.

To show that $\bar f = [f]_\vp$, we first claim that if $h \in \frakp_i^2$, then $[h]_\vp = 0$. Indeed, taking a directional derivative with respect to $x_j$ is the same as to taking a directional derivative with respect to $(x_{j,1}, z_j)$. If $i \neq j$, this sends elements of $\frakp_i^2$ to elements of $\frakp_i^2$, while for $i = j$, this sends elements of $\frakp_i^2$ to elements of $\frakp_i$. Thus
\[\left.\frac{\partial}{\partial t_1} \cdots \frac{\partial}{\partial t_d} h(x_1 + t_1y_1, \ldots, x_d + t_dy_d)\right|_{t_1 = \cdots = t_d = 0} \in \frakp_i,\]
and evaluating at $\vx = \vp$ yields $[h]_\vp = 0$.

As a result, we have $[f]_\vp - [\bar f]_\vp = [f-\bar f]_\vp = 0$, so $[f]_\vp = [\bar f]_\vp$. However, since $\bar f\in \calm_d$, we have $[\bar f]_\vp = \bar f$. This concludes the proof.
\end{proof}

For the following section, we will also need a strange-looking auxiliary result. To state it, define the ring $\cals = \calr_{d,\vp}[y_1,\ldots,y_d]$, where $y_1,\ldots,y_d$ are free scalar variables. By assigning $y_i$ to have the same degree as $x_i$ (i.e.\ $e_i=(0,\ldots,0,1,0,\ldots,0)$), this remains a $\setz^d$-graded ring. Define the $\setf$-linear degree-preserving map $Y\colon \calm_{d,\vp} \to \cals$ given by
\[Y(f) = \sum_{S \subseteq [d]} f[\vp_S] \cdot \prod_{i \in S} y_s.\]
Finally, let $\frakq \subseteq \cals$ be the ideal generated by homogeneous elements that vanish when evaluated at $x_i = p_i$ and $y_i = -1$.
\begin{lem} \label{lem:yprops}
If $f \in \calm_{d,\vp}$, then $Y(f) \in \frakq^d$. If furthermore $f \in \frakp_i^2$ for some $i$, then $Y(f) \in \frakq^{d+1}$.
\end{lem}
\begin{proof}
Assume that $p_i = e_1$ for all $i$. Writing $f = x_{1,1} \cdots x_{d,1} g(z_1,\ldots,z_d)$ and $y'_i = x_{i,1} + y_i \in \frakq$, we have
\begin{align*}
Y(f) &= \sum_{S \subseteq [d]} \paren*{\prod_{i\in [d]\setminus S} x_{i,1} \cdot \prod_{i \in S} y_i \cdot g(\vz_{S \to 0})} \\
&= \sum_{S \subseteq [d]} \paren*{\prod_{i\in [d]\setminus S} x_{i,1} \cdot \prod_{i \in S} (y'_i - x_{i,1}) \cdot g(\vz_{S \to 0})} \\
&= \sum_{S' \subseteq S \subseteq [d]} \paren*{(-1)^{\abs{S} - \abs{S'}}\prod_{i\in [d]\setminus S'} x_{i,1} \cdot \prod_{i \in S'} y'_i \cdot g(\vz_{S \to 0})} \\
&= \sum_{S' \subseteq [d]} \paren*{\prod_{i\in [d]\setminus S'} x_{i,1} \cdot \prod_{i \in S'} y'_i \cdot \sum_{S \supseteq S'} (-1)^{\abs{S} - \abs{S'}} g(\vz_{S \to 0})}.
\end{align*}
where $\vz_{S \to 0}$ denotes $(z_1,\ldots,z_d)$ but with $z_i$ replaced with $0$ for all $i \in S$. 
For each given $S'$, we have $\prod_{i \in S'} y'_i \in \frakq^{\abs{S'}}$, since $y'_i \in \frakq$, and the rational function
\[g_{S'} \coloneq \sum_{S \supseteq S'} (-1)^{\abs{S}} g(\vz_{S \to 0})\]
is such that it evaluates to zero when any $z_i$ with $i \notin S'$ is set to zero. Hence it lies in $\prod_{i \notin S'} \frakp_i \subseteq \frakq^{d - \abs{S'}}$.
This shows that $Y(f) \in \frakq^d$.

If $f \in \frakp^2_i$, then we have $g(\vz_{S \to 0})$ is in $\frakp_i^2$ for all $S$ and zero if $i \in S$. Therefore, if $i \in S'$, we have $g_{S'} = 0$, and if $i \notin S'$, we have $g_{S'} \in \frakp_i^2 \cap \prod_{i \notin S'} \frakp_i = \frakp_i \cdot \prod_{i \notin S'} \frakp_i \subseteq \frakq^{d + 1 - \abs{S}}$. Thus $Y(f) \in \frakq^{d+1}$.
\end{proof}

\subsection{Approximating matrices}
Since $[-]_\vp$ is $\setf$-linear, for an $\calm_{d,\vp}$-matrix $M$ we may now construct $[M]_\vp$, recalling that this means applying $[-]_\vp$ to $M$ entrywise.  We prove two results about this operation. Both these bounds are tight (see \cref{ex:tight}).
\begin{prop} \label{lem:mppr}
If $u$ and $v$ are homogeneous $\calr_{d,\vp}$-vectors such that $u \otimes v$ is an $\calm_{d,\vp}$-matrix, then $\PR([u \otimes v]_\vp) \leq 2^d$.
\end{prop}
\begin{proof}
Again, assume that $p_i = e_1$ for all $i$ and let $m$ be the invertible element $x_{1,1}x_{2,1} \cdots x_{d,1}$. Since there exist invertible elements in $\calr_{d,\vp}$ of all degrees, we may assume that $u$ has degree $(1,1,\ldots,1)$ and $v$ has degree $(0,0,\ldots,0)$. Let $w = mv$, so that $[u \otimes v]_\vp = [m\inv(u \otimes w)]_\vp$. Since $\sum_{i\in [d]} \frakp_i^2$ is an ideal, we have $[m\inv(u \otimes w)]_\vp = [m\inv([u]_\vp \otimes [w]_\vp)]_\vp$, so we may assume that $u$ and $w$ are in $\calm_d$.

We can now explicitly describe $[m\inv(u \otimes w)]_\vp$. We may uniquely decompose
\[u = \sum_{S \subseteq [d]} u_S \prod_{i \notin S} x_{i,1} \quad\text{and}\quad w = \sum_{S \subseteq [d]} w_S \prod_{i \notin S} x_{i,1},\]
where $u_S$ and $w_S$ are $\calm_S$-vectors that do not involve $x_{i,1}$ for any $i \in S$. Observe
\[\brac*{\frac{u_S \prod_{i \notin S} x_{i,1} \otimes  w_{S'} \prod_{i \notin S'} x_{i,1}}{m}}_\vp = \brac*{\frac{u_S \otimes w_{S'}}{\prod_{i \in S \cap S'} x_{i,1}} \cdot \prod_{i \notin S \cup S'} x_{i,1}}_\vp.
\]
If $S \cap S'$ contains an element $j$, then every entry of $u_S \otimes w_{S'}$ is in $\frakp_j
^2$ and applying $[-]_\vp$ yields zero. Otherwise, the quantity $u_S \otimes w_{S'} \cdot \prod_{i \notin S \cup S'} x_{i,1}$ is already an $\calm_d$-matrix and approximation does nothing. Summing over $S$ and $S'$, we conclude
\[\brac*{\frac{u \otimes w}{m}}_\vp = \sum_{S \cap S' = \emptyset} u_S \otimes w_{S'}  \prod_{i \notin S \cup S'} x_{i,1} = \sum_{S \subseteq [d]} u_S \otimes \paren*{\sum_{S \cap S' = \emptyset} w_{S'} \prod_{i \notin S \cup S'} x_{i,1}},\]
which indeed has partition rank at most $2^d$.
\end{proof}
\begin{prop} \label{lem:mprank}
If $M$ is an $\calm_{d,\vp}$-matrix, then $\rank([M]_\vp) \leq \sum_{S \subseteq [d]} \rank M[\vp_S]$.
\end{prop}

\begin{proof}
We will show
\[\sum_{S \subseteq [d]} \rank M[\vp_S] \stackrel{(1)}{\geq} \rank Y(M) \stackrel{(2)}\geq \rank Y([M]_\vp) \stackrel{(3)}{=} \rank([M]_\vp).\]
The inequality (1) is immediate from the subadditivity of rank.

To show (2), let $r = \rank Y([M]_\vp)$. There exists an $r \times r$ submatrix of $Y([M]_\vp)$ with nonzero determinant. Since this determinant is a polynomial in the $x_i$ and $y_i$ of degree exactly $dr$, its order of vanishing at $x_i = p_i$ and $y_i = -1$ is at most $dr$, meaning that it does not lie in $\frakq^{dr+1}$. By \cref{lem:yprops}, every entry in $Y(M)$ is in $\frakq^d$ and every entry of $Y(M) - Y([M]_\vp)$ is in $\frakq^{d+1}$. As a result, the corresponding minor of $Y(M)$ is the same modulo $\frakq^{dr+1}$ and is thus nonzero. Hence $\rank Y(M) \geq r$.

To show (3), consider the injective homomorphism $Y' \colon \calr_d \to \cals$ given by sending $x_{i,j} \mapsto x_{i,j} + p_{i,j}y_{i}$. It is easy to see that this agrees with $Y$ on $\calm_d$, so
\[\rank([M]_\vp)  = \rank(Y'([M]_\vp)) = \rank(Y([M]_\vp)).\qedhere\]
\end{proof}

\section{Schur Complementation} \label{sec:sc}
\subsection{Schur complements over rings}
In this section we recall the notion of Schur complement from linear algebra and extend it to $R$-matrices.

Let $R$ be an integral $\setf$-algebra, and let $M$ be an $R$-matrix. 
If $A$ is an invertible (square) submatrix of $M$, we define the Schur complement $M/A$ as follows: if $A$ is the top-left submatrix of $M$ and
\[M = \begin{pmatrix} A & B \\ C & D \end{pmatrix},\]
we define
\[M/A = \begin{pmatrix} 0 & 0 \\ 0 & D - BA\inv C\end{pmatrix}.\]
If $A$ is not the top-left submatrix of $M$, there exist permutation matrices $P$ and $Q$ such that $A$ is the top-left submatrix of $PMQ$, and we let $M/A = P\inv(PMQ/A)Q\inv$. It is straightforward to show that this construction is independent of the choice of $P$ and $Q$.
\begin{rmk}
This notation differs somewhat from the standard approach to Schur complements (appearing in the introduction of this paper), where it is assumed that $A$ is the top-left submatrix, which allows the zero rows and columns to be omitted from the definition of $M/A$. This difference is purely cosmetic, allowing us to avoid assuming (without loss of generality) that $A$ is the top-left submatrix.
\end{rmk}
\begin{rmk} \label{rmk:scbi}
The above definition, while concrete, is perhaps unsatisfying as it does not explain how Schur complements behave under changes of basis, which the following coordinate-free approach aims to remedy. Given a bilinear map $B \colon \setf^a \times \setf^b \to R$ and $r$-dimensional subspaces $U \subseteq \setf^a$ and $V \subseteq \setf^b$, one can restrict $B$ to get a map $B' \colon U \times V \to R$. If $B'$ has unit determinant, there exist unique projection maps $P \colon \setf^a \to U \otimes R$ and $Q \colon \setf^b \to V \otimes R$ such that $\tilde B'(P(x), y) = B(x, y)$ for all $x \in \setf^a$ and $y \in V$ and $\tilde B'(x, Q(y)) = B(x, y)$ for all $x \in U$ and $y \in \setf^b$, where $\tilde B'\colon (U \otimes R) \times (V \otimes R) \to R$ is the unique $R$-bilinear extension of $B'$. Then, the Schur complement $B/B'$ is given by \[(B/B')(x,y) = B(x,y) - \tilde B'(P(x),Q(y)).\] 
In particular, the Schur complement of a bilinear form depends only on the choice of subspaces $U$ and $V$. This is not obvious from the block definition, as a na\"ive translation would entail a decomposition $\setf^a = U \oplus U'$ and $\setf^b = V \oplus V'$. However, the choices of $U'$ and $V'$ do not matter.
\end{rmk} 

If $M$ is an $R$-matrix and $A$ is an $r \times r$ invertible submatrix, the Schur complement obeys the following basic properties:
\begin{itemize}
\item If $\varphi\colon R \to S$ is a homomorphism, then $\varphi(M/A) = \varphi(M)/\varphi(A)$. In particular, the choice of base ring does not affect the Schur complement.
\item $\rank(M/A) = \rank(M) - r$. This comes from the identity
\[\begin{pmatrix} I & 0 \\ -CA\inv & I \end{pmatrix}\begin{pmatrix} A & B \\ C & D \end{pmatrix}\begin{pmatrix} I & -A\inv B \\ 0 & I \end{pmatrix} = \begin{pmatrix} A & 0 \\ 0 & D - CA\inv B \end{pmatrix},\]
where the right-hand side has rank given by $\rank (M/A) + r$.
\item There exist $R$-vectors $u_1,\ldots,u_r,v_1,\ldots,v_r$ such that $M - M/A = u_1\otimes v_1 + \cdots + u_r \otimes v_r$. This is because in the case where $A$ is the top-left submatrix, we have
\begin{equation} \label{eq:mma} M - M/A = \begin{pmatrix} A & B \\ C & CA\inv B \end{pmatrix} = \begin{pmatrix} A \\ C \end{pmatrix} A\inv \begin{pmatrix} A & B \end{pmatrix}.\end{equation}
\end{itemize}

\subsection{Differential Schur complements}
Let $M$ be an $\calm_d$-matrix, $A$ be an $r\times r$ submatrix of $M$, and $\vp \in (\setf^n)^d$ be such that $A(\vp)$ is invertible. Since $A(\vp)$ is invertible, the quantity $\det A$ is homogeneous of degree $(r,r,\ldots,r)$ in $\calr_d$ and nonvanishing at $\vp$, and is thus invertible in $\calr_{d,\vp}$. Therefore $A$ is invertible in $\calr_{d,\vp}$ and we may take the Schur complement $M/A$ in $\calr_{d,\vp}$. Since $AA\inv = I$ has degree zero, $A\inv$ is of degree $(-1,-1,\ldots,-1)$, implying that $M/A$ is an $\calm_{d,\vp}$-matrix.
\begin{defn} \label{def:dsc}
The \vocab{differential Schur complement} is the $\calm_d$-matrix $[M/A]_{\vp}$, the result of taking the Schur complement in the ring $\calm_{d, \vp}$, followed by applying the map $[-]_\vp$ entrywise.
\end{defn}

Applying \cref{lem:mppr,lem:mprank} yields the following results:
\begin{prop} \label{lem:mappr}
$\PR(M - [M/A]_\vp) \leq 2^d r$.
\end{prop}
\begin{proof}
By \labelcref{eq:mma}, we may write
\[M - M/A = \sum_{i=1}^r u_i \otimes v_i\]
where for each $i \in [r]$, $u$ and $v$ are homogeneous $\calr_{d,\vp}$-vectors such that $u \otimes v$ is an $\calm_{d,\vp}$-matrix. Applying \cref{lem:mppr}, we conclude that
\[\PR(M - [M/A]_\vp) = \PR([M - M/A]_\vp) = \PR\paren*{\sum_{i=1}^r [u_i \otimes v_i]_\vp} \leq \sum_{i=1}^r \PR([u_i \otimes v_i]_\vp) \leq 2^d r.\qedhere.\]
\end{proof}

\begin{prop} \label{lem:maprank}
$\rank([M/A]_\vp) \leq \sum_{S \subseteq [d]} (\rank M[\vp_S] - r)$.
\end{prop}
\begin{proof}
Since partial evaluation is a homomorphism, we have
\[(M/A)[\vp_S] = M[\vp_S]/A[\vp_S],\]
which has a rank of $\rank M[\vp_S] - r$. Thus, by \cref{lem:mprank}, we have
\[\rank([M/A]_\vp) \leq \sum_{S \subseteq [d]} (M/A)[\vp_S] = \sum_{S \subseteq [d]} (\rank M[\vp_S] - r).\qedhere\]
\end{proof}

\subsection{Examples}
For each of the examples in this section, $\vp=(e_1,\ldots,e_1)$, 
and $\alpha_i$, $\beta_i$, and $\gamma_i$ be abbreviations for $x_{1,i}, x_{2,i}$, and $x_{3,i}$, respectively.
\begin{examp}[$d = 1$]
Suppose $M$ is an $\calm_1$-matrix such that there exist compatible block decompositions
\[M = \begin{pmatrix} A & B \\ C & D \end{pmatrix}\quad\text{and}\quad M(e_1) = \begin{pmatrix} I & 0 \\ 0 & 0 \end{pmatrix}.\]
Then, since $B(e_1)$ and $C(e_1)$ are zero, every entry in $CA\inv B$ is in $\frakp_1^2$ and thus is approximated by $0$. It follows that $[M/A]_\vp$ is very simple:
\[[M/A]_\vp = \begin{pmatrix}
0 & 0 \\
0 & D
\end{pmatrix}.\]
\end{examp}
\begin{examp}[$d = 2$]
Suppose $M$ is an $\calm_2$-matrix such that
\[M = \begin{pmatrix} A & B \\ C & D \end{pmatrix}\quad\text{and}\quad M(e_1,e_1)= \begin{pmatrix} I & 0 \\ 0 & 0 \end{pmatrix}.\]
Since $B$ and $C$ contain no $\alpha_1\beta_1$ terms, we may write $B = \beta_1 B_1 + \alpha_1 B_2 + B_{12}$ and $C = \beta_1C_1 + \alpha_1 C_2 + C_{12}$, where $B_1$ and $C_1$ are $\calm_{\set{1}}$-matrices with no $\alpha_1$ terms, $B_2$ and $C_2$ are $\calm_{\set{2}}$-matrices with no $\beta_1$ terms, and $B_{12}$ and $C_{12}$ are $\calm_{\set{1,2}}$-matrices with no $\alpha_1$ or $\beta_1$ terms. The only terms in the expansion of $CA\inv B$ that do not immediately go to zero after approximation are $\alpha_1\beta_1 C_1A\inv B_2 + \alpha_1\beta_1 C_2A\inv B_1$.

Note that if $E$ is a matrix of degree $(0,\ldots,0)$ with entries in $\frakp_1 + \frakp_2$, we have $[C_1 E B_2]_\vp = 0$, since all entries in $C_1$ are in $\frakp_1$ and all entries in $B_2$  are in $\frakp_2$. Moreover, $\alpha_1\beta_1 A\inv - I$ is such a matrix, as it vanishes at $(e_1,e_1)$. Hence, $[\alpha_1\beta_1 C_1A\inv B_2]_\vp = [C_1IB_2]_\vp = C_1B_2$. 
Similarly, $[\alpha_1\beta_1 C_2A\inv B_1]_\vp = C_2B_1$. Therefore
\[[M/A]_\vp = \begin{pmatrix}
0 & 0 \\
0 & D - C_1B_2 - C_2B_1
\end{pmatrix}.\]
\end{examp}

While computing $[CA^{-1} B]_\vp$ for $d=2$ can be reduced to computing $[CB/(\alpha_1\beta_1)]_\vp = C_1B_2 + C_2B_1$ (a product rule of sorts), for $d\ge 3$ the situation is more complicated.

\begin{examp}[$d = 3$]
Suppose $M$ is an $\calm_3$-matrix such that
\[M = \begin{pmatrix} A & B \\ C & D \end{pmatrix}\quad\text{and}\quad M(e_1,e_1,e_1) = \begin{pmatrix} I & 0 \\ 0 & 0 \end{pmatrix}.\]
If we follow the same procedure the $d = 2$ case, we will eventually wish to compute terms such as $[\alpha_1\beta_1\gamma_1^2 C_1 A\inv B_2]_\vp$ for $C_1$ an $\calm_{\set{1}}$-matrix and $B_2$ an $\calm_{\set{2}}$-matrix. However, this depends nontrivially on $A\inv$, as, say, a $\gamma_2/\gamma_1$ term affects the resulting approximation. While it is possible to write down an explicit formula for $[M/A]_\vp$, using a power series expansion of $A\inv$, it is sufficiently complicated that we will stop here.
\end{examp}

\begin{examp} \label{ex:tight}
The bounds in \cref{lem:mappr,lem:maprank} are tight. For notational simplicity we will work in the case $d = 2$, but a straightforward generalization of the following example exists for all $d$.

Let $M$ be the $5 \times 5$ $\calm_2$-matrix given by
\[M = \begin{pmatrix}
\alpha_1\beta_1 & \alpha_1\beta_1 & \alpha_1\beta_2 & \alpha_2\beta_1 & \alpha_2\beta_2 \\
\alpha_1\beta_1 & 0 & 0 & 0 & 0 \\
\alpha_1\beta_2 & 0 & 0 & 0 & 0 \\
\alpha_2\beta_1 & 0 & 0 & 0 & 0 \\
\alpha_2\beta_2 & 0 & 0 & 0 & 0 
\end{pmatrix}.\]
If we let $A$ be the $1 \times 1$ top-left submatrix of $M$, we may compute
\[[M/A]_\vp = \begin{pmatrix}
0 & 0 & 0 & 0 & 0 \\
0 & -\alpha_1\beta_1 & -\alpha_1\beta_2 & -\alpha_2\beta_1 & -\alpha_2\beta_2 \\
0 & -\alpha_1\beta_2 & 0 & -\alpha_2\beta_2 & 0 \\
0 & -\alpha_2\beta_1 & -\alpha_2\beta_2 & 0 & 0 \\
0 & -\alpha_2\beta_2 & 0 & 0 & 0 
\end{pmatrix}.\]
It is clear that $\rank ([M/A]_\vp) = 4$. Moreover, $\PR(M - [M/A]_\vp) = 4$ as well; indeed, the first two rows of $M-[M/A]_\vp$ are identical, so its partition rank 
is at most $4$,
and moreover, its $4 \times 4$ bottom-right submatrix has rank---and thus partition rank---at least $4$.
This matches the bounds in \cref{lem:mappr,lem:maprank}. In particular, this means that \cref{lem:mppr,lem:mprank} are also tight.
\end{examp}

\section{Proof of \texorpdfstring{\cref{thm:main}}{Theorem \ref{thm:main}}} \label{sec:mainproof}

\subsection{Polynomial methods}
For a polynomial $f(x) = f(x_1,\ldots,x_n)$ over $\setf$ and a point $p \in \setf^n$,  
let $\mult(f,p)=\mindeg f(x+p)$, the minimum degree of any monomial in the support of $f(x+p)$.
Note that $f(p)=0$ if and only if $\mult(f,p) \ge 1$.
As a simple example,
\[\mult(x_1x_2,(p_1,p_2)) 
= \mindeg((x_1+p_1)(x_2+p_2))
= \begin{cases} 2 & (p_1,p_2)=(0,0)\\ 1& (p_1,p_2) \neq (0,0), p_1p_2 = 0\\ 0 & p_1p_2 \neq 0 \end{cases}.\]

The following is a generalization of the Schwartz-Zippel lemma, appearing in \cite{DKSS13}. We include an elementary proof for completeness.
\begin{lem}[\cite{DKSS13}]\label{lem:multSZ}
If $f \in \setf[x_1,\ldots,x_n]$ is nonzero of total degree $d$ and $S \subseteq \setf$ is a finite subset, then
\[\sum_{p \in S^n} \mult(f, p) \leq d\cdot \abs{S}^{n-1}.\]
\end{lem}
\begin{proof}
We apply a double induction on $n$ and $\deg f$. The $n = 0$ case is clear as $f$ is constant and the left-hand side is zero. For some $n > 0$, the $d = 0$ case is obvious. If $d > 0$, then first suppose that $f$ is not divisible by $x_n - s$ for any $s \in S$. Then, we have
\[\sum_{p \in S^n} \mult(f, p) = \sum_{s \in S} \sum_{p' \in S^{n-1}} \mult(f, (p', s)) \leq \sum_{s \in S} \sum_{p' \in S^{n-1}} \mult(f(-, s), p') \leq \sum_{s \in S} d \cdot \abs{S}^{n-2} = d \cdot \abs{S}^{n-1},\]
where we have used the inductive hypothesis on the nonzero polynomials $f(-, s)$ for $s \in S$ (which have total degree at most $d$).

Otherwise, we have $f = (x_n - s)g$ for some $g$. In this case, we have
\[\sum_{p \in S^n} \mult(f, p) = \sum_{p \in S^n} \mult(x_n - s, p) + \mult(g, p) \leq \abs{S}^{n-1} + (d-1) \cdot \abs{S}^{n-1} = d \cdot \abs{S}^{n-1},\]
where we have used the inductive hypothesis on $g$.
\end{proof}

We are now ready to prove the first half of \cref{thm:main}.
\begin{lem} \label{lem:ercrint}
Suppose $\setf = \setf_q$ and $M$ is an $\calm_d$-matrix for some $d \geq 1$. Then
\[\sete_{p_d \in \setf^n} \CR(M[p_d]) \geq \paren*{1 - \frac{1}{q}}\CR(M).\]
\end{lem}
\begin{proof}
Let $\CR(M) = r$. Replacing $M$ with an $r \times r$ submatrix of full commutative rank does not change the right-hand side, while the left-hand side cannot increase. Thus we may assume $M$ is $r \times r$.

The polynomial $\det M$ is nonzero and is homogeneous of degree $(r,r,\ldots,r)$. In particular, for every monomial $m=m(x_1,\ldots,x_{d-1})$ in $\calr_{d-1}$ of degree $(r,r,\ldots,r)$, there is some degree-$r$ polynomial $f_m(x_d) \in \calr_{\set{d}}$ such that $\det M = \sum_m f_m \cdot m$.
Pick some $m$ such that $f_m$ is nonzero.

Let $p_d \in \setf^n$.
Note that by the multilinearity of the determinant we have
\[\det M[x_d + p_d] = \det (M[x_d] + M[p_d]) = \sum_{S \subseteq [r]} \det M_S,\]
where $M_S$ is the ``mixed matrix'' where the rows indexed by $S$ come from $M[p_d]$ and the rest come from $M[x_d]$. If $\abs{S} > \CR(M[p_d])$, then the rows from $M[p_d]$ are linearly dependent, so $\det M_S = 0$. As a result, every monomial in $\det M[x_d + p_d]$ has degree at least $r - \CR(M[p_d])$ in $x_d$, 
meaning that $\mult(f_m, p) \geq r - \CR(M[p_d])$. 

Applying \cref{lem:multSZ} with $f_m$, which is nonzero of degree $r$, and with $S = \setf$, yields
\[r - \sete_{p_d}\CR(M[p_d]) \leq \sete_{p_d} [\mult(f_m, p_d)] \leq r/q,\]
so $\sete_{p_d}\CR(M[p_d]) \geq (1-1/q)r$, as desired.
\end{proof}
\begin{cor} \label{lem:ercrnew}
Suppose $\setf = \setf_q$ and $M$ is an $\calm_d$-matrix. Then
\[\MaxR(M) \geq \sete_{\vp \in (\setf^n)^d} [\rank M(\vp)] \geq \paren*{1-\frac{1}{q}}^d\CR(M).\]
\end{cor}
\begin{proof}
The left inequality is trivial as $\MaxR(M) \geq \rank M(\vp)$ always. To prove the right inequality, repeatedly apply \cref{lem:ercrint} to get
\begin{multline*} \sete_{\vp_{[d]}} [\CR(M[\vp_{[d]}])] \geq \paren*{1-\frac{1}{q}} \sete_{\vp_{[d-1]}}[\CR(M[\vp_{[d-1]}])] \geq \cdots \\ \geq \paren*{1-\frac{1}{q}}^{d-1} \sete_{\vp_{[1]}}[\CR(M[\vp_{[1]}])] \geq \paren*{1-\frac{1}{q}}^d \CR(M). \end{multline*}
The leftmost quantity is simply $\sete_{\vp}[\rank M(\vp)]$, so we are done.
\end{proof}
\begin{rmk}
If $d > 0$ and $M \neq 0$, we in fact have $\MaxR(M) > (1-1/q)^d\CR(M)$. This is because $\rank M(0) = 0 < \MaxR(M)$, which means that $\rank M(\vp)$ is not a constant random variable and thus $\MaxR(M) > \sete_{\vp}[\rank M(\vp)]$.
\end{rmk}
\begin{prop} \label{prop:tight}
The constant $(1-1/q)^d$ cannot be improved.
\end{prop}
\begin{proof}
In the case $d = 1$, we may construct for any $k\ge 1$ the diagonal matrix $M_k$ of size $(q^k-1)/(q-1)$ whose diagonal entries consist of all the nonzero linear forms in $x_{1,1},\ldots,x_{1,k}$ up to scaling. Then $\CR(M_k) = (q^k-1)/(q-1)$. Moreover, observe that, 
\[\rank M_k(p) = \begin{cases} q^{k-1} & (p_1,\ldots,p_k) \neq 0 \\ 0 & \text{else} \end{cases} ,\]
so $\MaxR(M_k) = q^{k-1}$. Thus,
$\MaxR(M_k)/\CR(M_k) = (1 - 1/q^k)/(1 - 1/q)$, 
so $\lim_{k \to \infty} \MaxR(M_k)/\CR(M_k) = 1-1/q$.

For $d \ge 2$, consider the $\calm_d$-matrix $M_{k,d} = M_k(x_1) \boxtimes \cdots \boxtimes M_k(x_d)$, where $\boxtimes$ denotes the Kronecker product of matrices.
Then $\MaxR(M_{k,d}) = \MaxR(M_k)^d$ and $\CR(M_{k,d}) = \CR(M_k)^d$, so $\lim_{k \to \infty} \MaxR(M_{k,d})/\CR(M_{k,d}) = (1-1/q)^d$.
\end{proof}

\begin{lem} \label{lem:prcrl}
Let $M$ be an $\calm_d$-matrix. If $\setf$ is infinite, then $\PR(M) \leq 2^d \CR(M)$. If $\setf = \setf_q$, then
\[\PR(M) \leq \frac{2^d(q-1)^d}{(2^d-1)(q-1)^d-(2^d-2)q^d} \CR(M),\]
provided the denominator is positive (which is true for sufficiently large $q$ in terms of $d$).
\end{lem}
One can show that $q \geq d2^d$ suffices.

\begin{proof}
If $d = 0$ the result is obvious, so assume $d \ge 1$.

Let $C=C(d,q)$ be the constant from the statement. We note that
\begin{equation}\label{eq:C-rel}
    Cq^d = 2^d(q-1)^d + C(2^d-1)(q^d - (q-1)^d).
\end{equation}
In particular, $Cq^d \geq 2^d(q-1)^d + C(q^d - (q-1)^d)$, 
which rearranges to $C \geq 2^d$.

We use induction on $k = \CR(M)$. The case $k = 0$ is trivial, so assume $k>0$. Choose some $\vp \in (\setf^n)^d$ such that $r = \rank M(\vp) = \MaxR(M)$. If $\setf$ is infinite, then $r = k$. If $\setf = \setf_q$, then by \cref{lem:ercrnew} we have $r \geq (1 - 1/q)^d k$. Let $A$ be an $r \times r$ submatrix of $M$ such that $A(\vp)$ is invertible. Then, we have
\[\PR(M) \leq \PR(M - [M/A]_\vp) + \PR([M/A]_\vp)
\le 2^d r + \PR([M/A]_\vp),\]
using \cref{lem:mappr}.
Moreover, by \cref{lem:maprank},
\[\CR([M/A]_\vp) \leq \sum_{S \subseteq [d]} (\rank M[\vp_S] - r) = \sum_{S \subsetneq [d]} (\rank M[\vp_S] - r) \leq (2^d-1)(k-r).\]
If $\setf$ is infinite, this means that $[M/A]_\vp = 0$, so we are done. In the finite case, we have
\[(2^d - 1)(k-r) \leq (2^d - 1) \cdot \frac{q^d - (q-1)^d}{q^d} \cdot k\]
which is less than $k$ if $(2^d-1)(q-1)^d > (2^d-2)q^d$. We use the inductive hypothesis on $[M/A]_\vp$ to get
\[\PR(M) \leq 2^d r + C (2^d - 1)(k-r).\]
The right hand side is (weakly) monotone decreasing in $r$; indeed, the coefficient of $r$ is $2^d - C(2^d-1) \le 2^d - 2^d(2^d-1) \leq 0$.
Since $r/k \ge (1-1/q)^d$, we have
\[\frac{\PR(M)}{\CR(M)} \le \frac{2^d(q-1)^d + C (2^d - 1)(q^d - (q-1)^d)}{q^d} = C,\]
using \labelcref{eq:C-rel}. This completes the proof.
\end{proof}

\subsection{Very small fields}
\cref{lem:ercrnew} proves the first half of \cref{thm:main}, while \cref{lem:prcrl} shows the second half of \cref{thm:main} for sufficiently large fields. To handle the remaining cases, we pass to a field extension.

Suppose $\setf = \setf_q$ is a finite field and let $\setk = \setf_{q^k}$ be the degree-$k$ extension of $\setf$. Let $\calm_d^\setf$ and $\calm_d^\setk$ be $\calm_d$ considered over $\setf$ and $\setk$, respectively. Given an $\calm_d^\setf$-matrix $M^\setf$, it can also be considered as an $\calm_d^\setk$-matrix, which we call $M^\setk$.

Since commutative rank can be described in terms of the vanishing of minors, which doesn't depend on the base field, we have $\CR(M^\setk) = \CR(M^\setf)$. The following lemma is a variant of \cite[Proof of Cor.~1]{CM23}.

\begin{lem} \label{lem:ext}
$\PR(M^\setf) \leq k \PR(M^\setk)$.
\end{lem}
\begin{proof}
Let $\beta_1,\ldots,\beta_k$ be an $\setf$-basis of $\setk$ and pick an arbitrary $\setf$-linear projection $\varphi \colon \setk \to \setf$. We claim that for any $\calm_d^\setk$-matrix $M$, we have $\PR(\varphi(M)) \leq k\PR(M)$. To show this it suffices to consider $M$ of partition-rank one, i.e.\ $M$ of the form $u \otimes v$ for an $\calm^\setk_S$-vector $u$ and an $\calm^\setk_{[d]\setminus S}$-vector $v$. Writing $u = \sum_{i \in [k]} u_i \beta_i$ for $\calm^\setf_S$-vectors $u_i$, we therefore have
\[\varphi(M) = \sum_{i \in [k]} \varphi(u_i \otimes \beta_i v) = \sum_{i\in [k]} u_i \otimes \varphi(\beta_i v),\]
which indeed has partition rank at most $k$.
\end{proof}

\begin{proof}[Proof of \cref{thm:main}]
In light of \cref{lem:ercrnew,lem:prcrl}, it suffices to show that for any fixed finite field $\setf$, we have $\PR(M) \lsim_d \CR(M)$, where $\lsim_d$ denotes an inequality up to some constant depending on $d$.

Choose some $k$ such that \cref{lem:prcrl} holds for $\setk$ as above. Then, by \cref{lem:ext} we have
\[\PR(M^\setf) \lsim_d \PR(M^\setk) \lsim_d \CR(M^\setk) = \CR(M^\setf),\]
as desired.
\end{proof}

\section{Degree-1 Applications} \label{sec:d1}
\subsection{Partition rank and commutative rank}
\cref{thm:d1bound} is precisely the $d = 1$ case of \cref{lem:prcrl}. We give here an independent, shorter proof. Before we do so, we recall that, as mentioned in the introduction, the partition rank of $\calm_1$-matrices has the following natural interpretation:
\begin{fact}
For an $\calm_1$-matrix $M$, $\PR(M) \leq r$ if and only if there exist
invertible scalar matrices $P$ and $Q$ such that
\[PMQ = \begin{pmatrix}
M_{11} & M_{12} \\
M_{21} & 0
\end{pmatrix},\]
with $M_{11}$ of size $r_1 \times r_2$ with $r_1+r_2 \le r$.
\end{fact}
The reason for this is that given any way to write
\[M(x_1) = \sum_{i=1}^{r_1} u_i \otimes v_i(x_1) + \sum_{i=1}^{r_2} u'_i(x_1) \otimes v'_i\]
for scalar vectors $u_i$ and $v'_i$ and $\calm_1$-vectors $v_i$ and $u'_i$, we may change basis such that all the $u_i$ are in the subspace of $\setf^n$ generated by the first $r_1$ basis elements and all the $v_i$ are in the subspace of $\setf^n$ generate by the first $r_2$ basis elements. This corresponds exactly to the above block decomposition of $PMQ$.
\begin{proof}[Proof of \cref{thm:d1bound}]
Write $\alpha_i = x_{1,i}$. We use induction on $k = \CR(M)$, with the $k = 0$ case being trivial.

If $k > 0$, pick some $p$ such that $M(p)$ is nonzero and let $r = \rank M(p)$. Without loss of generality assume that $p = e_1$. If we change basis so that 
\[M(p) = \begin{pmatrix}
I & 0 \\ 0 & 0
\end{pmatrix},\]
we have
\[M = \begin{pmatrix}
\alpha_1 I + A & B \\
C & D
\end{pmatrix},\]
where $A$, $B$, $C$, and $D$ are $\calm_1$-matrices that do not depend on $\alpha_1$. We claim that $\CR(D) \leq k - r$, which will finish since we will have
\[\PR(M) \leq \PR\paren*{\begin{pmatrix}
\alpha_1 I + A & B \\
C & 0
\end{pmatrix}} + \PR(D) \leq 2r + 2(k-r) = 2k,\]
by the inductive hypothesis on $D$.

If $\CR(D) \ge k-r+1$, we may delete rows and columns of $M$ to yield a matrix
\[M' = \begin{pmatrix}
\alpha_1 I + A & B' \\
C' & D'
\end{pmatrix}\]
such that $D'$ is $(k-r+1) \times (k-r+1)$ and $\det D' \neq 0$. We claim that $\det M' \neq 0$, which will be a contradiction as it implies $\CR(M) \geq \CR(M') \ge k+1$.

The determinant of $M'$ is a polynomial in $\alpha_1,\ldots,\alpha_n$; we claim that the $\alpha_1^r$-coefficient, as a polynomial in $\alpha_2,\ldots,\alpha_n$, is exactly $\det D'$. Indeed, given any term in the expansion of $\det M'$, the only way for it to have degree at least $r$ in $\alpha_1$ is if it involves every nonzero entry in $\alpha_1 I$. The sum of all these terms yields $\alpha_1^r \det D'$ plus terms of lower degree in $\alpha_1$, proving the claim. Thus $\det M' \neq 0$, as desired.
\end{proof}
\subsection{\texorpdfstring{$3$}{3}-tensors}
In this section, we assume $\setf = \setf_q$.
We use our machinery to prove~\cref{cor:arsr}, giving the best bounds known for $\sup_T \SR(T)/\AR(T)$, where $T$ ranges over all $T \in \calm_3(\setf)$.

For a $3$-tensor $T \in \calm_3$,  
recall that $\mat{T}$ denotes the corresponding $\calm_1$-matrix,
which is the $n \times n$ matrix whose entries $m_{j,k}$ are linear forms given by $T(x,y,z) = \sum_{j,k=1}^n m_{j,k}(x)y_jz_k$.
We start with a lemma relating the analytic rank of a $3$-tensor $T$ to the average rank of $\mat{T}$ at random points in some (large) subspace.
It is a small adaptation of a lemma in~\cite{MZ22} (see also~\cite{Lam25}).

\begin{lem}[]\label{lem:ar-avg}
For any $3$-tensor $T$ and $a \ge 0$, 
there exists a subspace $U \subseteq \setf^n$ such that
\[a\sete_{x \in U}[\rank \mat{T}(x)] + \codim U \leq (a+1)\AR(T).\]
\end{lem}
\begin{proof}
    Let $B \colon (\setf^n)^2 \to \setf^n$ be the bilinear map corresponding to 
    $T$, namely, $T(x,y,z) = \sum_{k=1}^n B_k(x,y)z_k$.
    Let $Z = B^{-1}(0) \subseteq (\setf^n)^2$ be the zero set of $B$.
    We first show that
    \[\sete_{(x,y)\in Z} [\rank B(x,-)] \le \AR(T).\]
    Observe that $\Pr_{(x,y) \in Z}[x=a] = q^{n-\rank B(a,-)}/|Z|$.
    Therefore,
    \[\sete_{(x,y) \in Z} q^{\rank B(x,-)} 
    = \sum_{a \in \setf^n} \Pr_{(x,y) \in Z}[x=a] q^{\rank B(a,-)} 
    = \sum_{a \in \setf^n} q^n/|Z| = q^{2n}/|Z| = 1/\bias(T) = q^{\AR(T)}.\]
    To finish, apply Jensen's inequality on the convex function $z \mapsto q^z$.

    Note that we similarly have 
    $\sete_{(x,y)\in Z} \rank B(-,y) \le \AR(T)$.
    Now, by linearity of expectation, 
    \[\sete_{(x,y)\in Z} [a\cdot\rank B(x,-) + \rank B(-,y)] 
    \le (a+1)\AR(T).\]
    Observe that $(x,y) \in Z$ if and only $x \in \ker B(-,y)$.
    Thus, there exists $y_0 \in \setf^n$ such that the subspace $U=\ker B(-,y_0)$ satisfies
    \[\sete_{x \in U}[a\cdot\rank B(x,-)] + \rank B(-,y_0) 
    \le (a+1)\AR(T).\]
    The proof now follows by noting that $\rank B(x,-) = \rank \mat{T}(x)$ and $\rank B(-,y_0) = \codim U$.
\end{proof}

\begin{proof}[Proof of~\cref{cor:arsr}]
    For any $\calm_1$-matrix $M$,
    \[\PR(M) \le 2\CR(M) \le \frac{2}{1-1/q}\sete_x[\rank M(x)]\]
    where the first inequality is~\cref{thm:d1bound},
    and the second inequality is~\cref{lem:ercrint}.

    Recall that $\SR(T) = \PR(T)  \le \PR(\mat{T})$.
    Apply~\cref{lem:ar-avg} with $T$ and $a=2/(1-1/q)$,
    and let $U \subseteq \setf^n$ be the resulting subspace.
    Then
    \begin{multline*}
        \SR(T) \le \SR(T\rvert_{U\times \setf^n \times \setf^n}) + \codim U
        \le \PR(\mat{T}\rvert_U) + \codim U
        \\ \le a\sete_{x \in U}[\rank \mat{T}(x)] + \codim U \le (a+1)\AR(T).
    \end{multline*}
    To finish the proof, note that $a+1 = \frac{3q-1}{q-1} = 3 + \frac{2}{q-1}$, as needed.
\end{proof}

\section{Relation to the Analytic Rank-Partition Rank Problem} \label{sec:arpr}
In this section, we assume $\setf = \setf_q$.

A major question in the study of tensor ranks is the following conjecture:
\begin{conj}[\cite{Lov19,KZ21,AKZ21}] \label{conj:arpr}
For any $d \geq 2$, there is a constant $C_d$ such that for any $d$-tensor $T$ over any finite field, $\PR(T) \leq C_d \AR(T)$.
\end{conj}
As it is not difficult to show that $\AR(T) \leq \PR(T)$, this conjecture is often described as the equivalence between the analytic and partition ranks. The case $d=2$ is trivial, as they are the same, while the case $d=3$ is resolved by \cref{thm:arsr}, which we improve slightly in \cref{cor:arsr}. In the case $d \geq 4$, the best known bound is due to previous work of the authors:
\begin{thm}[\cite{MZ22}] \label{thm:arprl}
$\PR(T) \leq C_d \AR(T) \cdot (1 + \log_{q}(1 + \AR(T)))$.
\end{thm}
In particular, this means that \cref{conj:arpr} holds for all fields with size at least polynomial in $\AR(T)$, so the log factor is only relevant for small fields.

Perhaps surprisingly, this logarithmic term has proven to be rather stubborn, as it appears to represent a fundamental gap in our knowledge of how pathologically polynomials can behave over small finite fields. Moreover, this gap appears to be specific to the partition rank, as the analytic rank over small fields has been connected to many other notions \cite{CY25,BL24}.

From this perspective, \cref{thm:main} appears to be the first result that breaks through this logarithmic barrier, establishing linear bounds on the partition rank over small finite fields, albeit in a special case. Indeed, we have $\AR(T) \leq \MaxR(\mat{T})$, meaning that matrices of low max-rank are a special case of tensors of low analytic rank. Specifically, given $T \in \calm_d$ and $\vp_{[d-2]} \in (\setf^n)^{d-2}$, the probability that a uniformly random $p_{d-1} \in \setf^n$ satisfies $T[\vp_{[d-1]}] = 0$ is exactly $q^{-\rank(\mat{T}[\vp_{[d-2]}])} \geq q^{-\MaxR(\mat{T})}$.  Since $\PR(T) \leq \PR(\mat{T})$, \cref{thm:main} shows \cref{conj:arpr} for tensors that have low analytic rank because their corresponding matrices have low max-rank.

Moreover, the techniques developed in this paper are highly connected to those in \cite{MZ22}. In \cite{MZ22}, we define the \vocab{local rank} of a tensor, and the proof of \cref{thm:arprl} proceeds by first proving that tensors of low analytic rank have low local rank, and then showing that all tensors of low local rank have low partition rank. The commutative rank of a multilinear matrix is a special case of local rank. (In the notation of \cite{MZ22}, we have $\LR_{0}(T) = (\CR(\mat{T}),0,\ldots,0)$; see \cite[Ex.~5.5]{MZ22}.)

If one were to apply the techniques of \cite[Sec.~7]{MZ22}, one would get the following:
\begin{prop}[Essentially a special case of {\cite[Thm.~1.5]{MZ22}}] \label{prop:prcrfull}
If $M$ is an $\calm_d$-matrix and $\CR(M) = \MaxR(M) = r$, then $\PR(M) \leq 2^dr$.
\end{prop}
However, this also follows from our theory of Schur complementation, since if $\CR(M) = \MaxR(M) = r$, we may find an $r \times r$ submatrix and some $\vp$ such that $A[\vp]$ is invertible. Taking the Schur complement with respect to this $A$ and $\vp$ yields that $[M/A]_\vp = 0$ and thus $\PR(M) = \PR(M - [M/A]_\vp) \leq 2^dr$. In fact, though they look very different, the arguments in \cite{MZ22} and this paper for this case are actually the same (in particular, they result in the same partition rank decomposition).

The main conceptual advance in this paper is the idea of decomposing a low-rank matrix/tensor in multiple steps, instead of all at once. The latter approach works over large fields, but relies on conditions (which we call \vocab{LR-stability} in \cite{MZ22}) that are too good to hope for over small fields. To address this, we instead decompose most of the matrix/tensor, and bound the rank of the remainder term. If this remainder has lower rank than the original tensor, iterating this process will eventually lead to a full decomposition.

This line of thought leads to a possible method for proving \cref{conj:arpr}. In \cite{MZ22}, we prove a vast generalization of \cref{prop:prcrfull} by constructing formulas that decompose all tensors of low local rank, under specific conditions. A natural question to ask is whether we can also use these formulas to decompose most of a tensor instead of all of it. In fact, this is possible, and leads to the notion of a ``multi-level Schur complement''.

The reason why this doesn't immediately prove \cref{conj:arpr} is the lack of an analogue of \cref{lem:ercrnew}.
In the proof of \cref{thm:main}, \cref{lem:ercrnew}
showed that it was always possible to take the Schur component with respect to a submatrix almost as large as the commutative rank, which was necessary to prevent the rank of the Schur complement from increasing uncontrollably. However, local rank is determined by polynomials of much higher degree, and as a result, it is impossible to show, via degree considerations alone, that almost all of the tensor can be decomposed in a single step. Developing an alternative approach to guarantee almost-total decomposition is the subject of ongoing work.
\printbibliography
\end{document}